\documentclass[a4paper,11pt]{article}

\usepackage[utf8]{inputenc}
\usepackage[english]{babel}
\usepackage[a4paper,twoside]{geometry} 
\usepackage{graphicx}
\usepackage{hyperref} 
\hypersetup{
    colorlinks = true,
    linkcolor = {blue},
	urlcolor = {red}
}
\usepackage{verbatim}

\usepackage{titlesec}
\usepackage{lmodern,lipsum}

\usepackage{emptypage}
\usepackage{amsmath}
\usepackage{amsthm}
\usepackage{amssymb}
\usepackage{amsfonts} 
\usepackage{mathrsfs}
\usepackage{enumerate}
\usepackage{mdwlist}
\usepackage{tikz}
\usepackage{tikz-cd}

\usepackage[T1]{fontenc}
\usepackage{color}
\definecolor{gray75}{gray}{0.75}
\newcommand{\hsp}{\hspace{20pt}}
\titleformat{\chapter}[hang]{\Huge\bfseries}{\thechapter\hsp\textcolor{gray75}{|}\hsp}{0pt}{\Huge\bfseries}

\usetikzlibrary{matrix}

    \titleformat{\section}
    {\normalfont\bfseries\Large}
    {\thesection.}{0.5em}{}

\newcounter{Results}[section]

\newtheorem{theorem}[Results]{Theorem}
\newtheorem{lemma}[Results]{Lemma}
\newtheorem{proposition}[Results]{Proposition}
\newtheorem{corollary}[Results]{Corollary}

\newtheorem{conjecture}[Results]{Conjecture}

\theoremstyle{remark}
\newtheorem{remark}[Results]{Remark}

\theoremstyle{definition}
\newtheorem{definition}[Results]{Definition}

\newcommand{\Z}{\ensuremath{\mathbb Z}}
\newcommand{\Q}{\ensuremath{\mathbb Q}}
\newcommand{\C}{\ensuremath{\mathbb C}}
\newcommand{\F}[1]{\ensuremath{\mathbb F_{#1}}}
\newcommand{\Oc}{\ensuremath{\mathscr{O}}}
\newcommand{\Tr}{\text{Tr}}
\newcommand{\Gal}{\text{Gal}}

\newcommand{\Cl}{\text{Cl}}
\newcommand{\im}{\text{im}}
\newcommand{\coker}{\text{coker}}

\title{Tame Galois module structure revisited}
\author{\textbf{Fabio Ferri} \\University of Exeter\\ \\\textbf{Cornelius Greither} \\ Universität der Bundeswehr München}

\begin{document}

\maketitle

\begin{abstract}
A number field $K$ is Hilbert-Speiser if all of its tame abelian extensions $L/K$
admit NIB (normal integral basis). It is known that $\Q$ is the only such field,
but when we restrict $\Gal(L/K)$ to be a given group $G$, the classification
of $G$-Hilbert-Speiser fields is far from complete. In this paper, we present
new results on so-called $G$-Leopoldt fields. In their definition, NIB is
replaced by ``weak NIB'' (defined below). Most of our results are negative,
in the sense that they strongly limit the class of $G$-Leopoldt fields
for some particular groups $G$, sometimes even leading to an
exhaustive list of such fields or at least to a finiteness result. In particular we are able to 
correct a small oversight in a recent article by Ichimura concerning Hilbert-Speiser fields.
\end{abstract}

\section{Introduction}

In the last decades a lot of work has been done 
on so-called Hilbert-Speiser fields; in particular, criteria were developed which
insure that a given field is not Hilbert-Speiser. The framework is Galois module structure of extensions of number fields.
More precisely the object of our study is the property of an extension of
number fields to have a \textit{normal integral basis}: an extension $L/K$ of number fields admits a NIB if $\Oc_L$ is a rank $1$ free
 $\Oc_K[G]$-module, where we denote by $\Oc_L$ and $\Oc_K$ the corresponding rings of integers. The Hilbert-Speiser theorem states that for abelian number fields tameness is equivalent to the existence of a normal integral basis over $\Q$, while 
 in general we only have that an extension with NIB has to be tame. 
By definition
a number field $K$ is Hilbert-Speiser if every tame abelian extension $L/K$ has NIB,
 and $K$ is $C_l$-Hilbert-Speiser if every such extension with Galois group 
isomorphic to $C_l$, the cyclic group of prime order $l$, has NIB.
 
 The first contribution to the topic of Hilbert-Speiser fields
came from the important result contained in \cite{GRRS99}: 
$\Q$ is the only Hilbert-Speiser field,
i.e. for every number field $K\supsetneq \Q$ there exists a (cyclic of prime order) tame abelian extension that does not have NIB. Subsequent research went towards
the finer problem of finding criteria for $C_l$-Hilbert-Speiser fields. For instance, in \cite{Car03}, \cite{Ich04} and \cite{Yos09}
we find some conditions for abelian $C_2$ and $C_3$-Hilbert-Speiser
fields. On the other hand, we know from \cite{Her05} and \cite{GJ09}  
that for arbitrary $l$, the prime $l$ cannot be
too highly ramified in $C_l$-Hilbert-Speiser fields.

In this work, we intend to consider a weakened version of NIB: 
we will say that a tame abelian extension $L/K$ of number fields has
a \textit{weak normal integral basis} (WNIB for short)
if $\mathcal{M}\otimes_{\Oc_K[G]}\Oc_L$ is free of rank $1$ over $\mathcal{M}$, where $\mathcal{M}$ is the maximal order of $K[G]$.
WNIB's have been studied for instance in \cite{Gre90}, \cite{Gre97} and \cite{GJ12}.
We shall ask the same
questions as above, substituting ``WNIB'' for ``NIB'' everywhere, and a number
field $K$ ``Leopoldt field'' (instead of ``Hilbert-Speiser field'')
if every tame abelian extension $L/K$ has WNIB.
We are mainly going to establish necessary conditions for number fields 
to be $C_l$-Leopoldt, where as before $l$ is a prime number. This
also gives criteria and (sometimes conditional) finiteness results 
for $C_l$-Hilbert-Speiser fields; for instance we will see that this permits to 
correct an oversight in the article \cite{Ich16}, whose techniques, 
even though they were originally
conceived to deal with Hilbert-Speiser fields, turn out to be supple enough 
to be applied to Leopoldt fields as well. 
Ichimura informs us that the corrected proof of Proposition 2 of \cite{Ich16},
to appear in a corrigendum, will proceed by a slightly different
argument Ichimura already knew at that time, but which finally was not
used in the paper  \cite{Ich16}. 

First of all we introduce some notation:

\begin{itemize}
 \item $\Cl(\Lambda)$ denotes the locally free class group $\Lambda$, where 
 $\Lambda$ is an order over a Dedekind domain (see for instance \cite[$\S$39A or $\S$49]{CR87})
 \item $h_K$ is the class number of the number field $K$
 \item $h_K^-$ is the relative class number of the CM-field $K$
 \item $h_l^-:=h_{\Q(\zeta_l)}^-$
 \item $h(d):=h_{\Q(\sqrt{d})}$
 \item $N_{L/K}$ stands for the norm map, defined either on the fractional ideals or directly between the ideal class groups of the number fields; namely the mutiplicative extension of the map $\mathcal{P}\mapsto\wp^{f(\mathcal{P}|\wp)}$ where $\mathcal{P}|\wp$ is an extension of prime ideals
 \item $H_K$ is the Hilbert class field of $K$
\end{itemize}

In the remainder of this introduction,
let us recall the concepts and formulate the questions we will
be dealing with in this paper, and review a few standard results
which are important in the sequel.

\begin{definition}
 Let $L/K$ be a tame $G$-Galois abelian extension of number fields, and let $\mathcal{M}$ be the maximal order of $K[G]$. The extension has a weak normal integral basis 
 (WNIB)
 if $\mathcal{M}\otimes_{\Oc_K[G]}\Oc_L$ is free of rank $1$ over $\mathcal{M}$.
\end{definition}

Without any essential changes to the classical proofs on normal integral bases, we have the following
\begin{proposition}\label{disjointwnib}
 Let $L/K$ be an abelian tame extension of number field and $F$ an intermediate field. If $L/K$ has a WNIB, so does $F/K$. 
 
 Let $L_1/K$ and $L_2/K$ be arithmetically disjoint extensions of number fields with $L_1/K$ tame abelian. If $L_1/K$ has WNIB, so does $L/L_2$, where $L=L_1L_2$.
\end{proposition}

We recall that there is the following surjection, with kernel denoted by $D(\Oc_K[G])$:
\begin{equation*}
 0\longrightarrow D(\Oc_K[G])\longrightarrow \textup{\Cl}(\Oc_K[G])\longrightarrow\textup{\Cl}(\mathcal{M})\longrightarrow 0.
\end{equation*}

\begin{definition}
 Let $K$ be a number field. We say that $K$ is Leopoldt if every tame abelian extension $L/K$ has WNIB, and
 $G$-Leopoldt fields if every tame abelian extension $L/K$ with Galois group isomorphic to $G$ has WNIB.
\end{definition}

\begin{remark}
  We note that other authors (see for example \cite{Car07} and \cite{BCGJ11}) employ the term
  ``Leopoldt field'' in a different (and stricter) sense, so that we briefly
  considered calling our definition ``weakly Leopoldt'';
  but we have decided against it for simplicity.
\end{remark}

As we stated in the abstract, $\Q$ is the only Hilbert-Speiser field. One may ask if this is still true in the weaker sense:
\begin{conjecture}
 $\Q$ is the only Leopoldt field.
\end{conjecture}
However, there is no chance to adapt the proof of \cite{GRRS99}. A fundamental ingredient is the study of objects related to
Swan modules, which live in the kernel group, and 
this approach cannot work in our context as we are interested in a quotient by the kernel group (that is the class group of the maximal order).
Instead we study $C_l$-Leopoldt fields directly.

In two short sections we will deal with the simple cases $l=2$ and $l=3$,
for the sake of completeness.  In Section \ref{real}
we develop necessary conditions for totally real fields that intersect $\Q(\zeta_l)$ trivially
to be $C_l$-Leopoldt. In Section \ref{hilbert-speiser} we
prove that the list of real abelian $C_l$-Hilbert-Speiser fields for certain values of $l$ given in \cite{Ich16} is actually complete.
The final Section \ref{intersection} presents some results concerning
fields that may intersect $\Q(\zeta_l)$ nontrivially.

In the rest of the introduction we review some basic tools.
We can firstly use McCulloh's important theorem \cite{McC83} to speak about realizable classes. Let us recall it.
First of all, the augmentation map $\varepsilon:\Oc_K[G]\rightarrow \Oc_K$
defines a map
\begin{equation*}\begin{split}
 \varepsilon_*:\Cl(\Oc_K[G])&\longrightarrow \Cl(\Oc_K)\\
 [M]&\longmapsto[\Oc_K\otimes_{\Oc_K[G]}M],\end{split}
\end{equation*}
where by $[M]$ we mean the class of $M$ in the locally free class group.
We define $\Cl^0(\Oc_K[G])$ to be the kernel of the above homomorphism. When $G$ is elementary abelian, we want to give the group $\Cl(\Oc_K[G])$ a structure of 
$\Z[\Delta]$-module, where $\Delta$ is the multiplicative group of the finite field whose additive structure is isomorphic to $G$:
if $G\cong \F{l^n}^+$, then $\Delta:=\F{l^n}^*$. This 
is simply done noting that each element of $\Delta$ defines a ring automorphism of $\Oc_K[G]$ and hence a functorial endomorphism of $\Cl(\Oc_K[G])$.
Now let
 \begin{equation*}
  \theta=\frac{1}{l}\sum_{g\in \Delta}t_{\F{l^n}/\F{l}}(g)g^{-1}\in\Q[\Delta]
 \end{equation*}
 be the Stickelberger element, where $t_{\F{l^n}/\F{l}}$ is the composition of the trace $\Tr_{\F{l^n}/\F{l}}$ with the integer representative of the elements of
 $\F{l}\cong\Z/l\Z$ between $0$ and $l-1$, and
 \begin{equation*}
  \mathcal{J}=\Z[\Delta]\theta\cap\Z[\Delta]\subseteq\Z[\Delta]
 \end{equation*}
 the Stickelberger ideal.
 
 \begin{theorem}[McCulloh \cite{McC83}]\label{mcculloh}
 Let $K$ be a number field and $G$ an elementary abelian group. Then
 \begin{equation*}
  R(\Oc_K[G])=\textup{\Cl}^0(\Oc_K[G])^{\mathcal{J}},
 \end{equation*}
 where we denote by $R(\Oc_K[G])$ the ``realizable classes'' in $\textup{\Cl}(\Oc_K[G])$, i.e. given by the tame Galois extensions of $K$ with Galois group $G$.
\end{theorem}
Using a cancellation property for projective modules
(which can easily be established over commutative rings), we obtain:
\begin{corollary}
 A number field $K$ is $G$-Hilbert-Speiser with $G$ elementary abelian if and only if the Stickelberger ideal $\mathcal{J}$ annihilates $\textup{\Cl}^0(\Oc_K[G])$.
\end{corollary}
This is of course the main ingredient for the research on Hilbert-Speiser fields that has been done.

Now let us return to our to our main focus of studying Leopoldt fields.
As before we can take $G$ to be any elementary abelian group and obtain
\begin{equation*}
 R(\mathcal{M}):=\left\{ \left[\mathcal{M}\otimes_{\Oc_K[G]}\Oc_L \right]:L/K \text{ } G \text{-tame}\right\}=\im\left(\Cl^0(\Oc_K[G])^{\mathcal{J}}\right)=
 \im\left(\Cl^0(\Oc_K[G])\right)^{\mathcal{J}}
\end{equation*}
(the action of $\Delta$ commutes with the map $\textup{\Cl}(\Oc_K[G])\rightarrow\textup{\Cl}(\mathcal{M})$). 
The extension of scalars to the maximal order does not change the augmentation, 
so from the surjectivity
we obtain that
\begin{equation*}
 R(\mathcal{M})=\Cl^0(\mathcal{M})^{\mathcal{J}}.
\end{equation*}

We recall that we know what the maximal order $\mathcal{M}$ is. If $K\times K(\zeta_l)\times\cdots\times K(\zeta_l)$ is the Wedderburn decomposition of
$K[G]$, then $\mathcal{M}$ is simply $\Oc_K\times \Oc_{K(\zeta_l)}\times\cdots\times \Oc_{K(\zeta_l)}$ and thus the part annihilated by
augmentation is $\Oc_{K(\zeta_l)}\times\cdots\times \Oc_{K(\zeta_l)}$.
Therefore we get the following
\begin{corollary}\label{leocond}
  Let $K$ be a number field and $G$ an $l$-elementary abelian group. Then $K$ is $G$-Leopoldt if and only if the Stickelberger ideal $\mathcal{J}\subseteq\Z[\Delta]$ 
  annihilates $\textup{\Cl}(K(\zeta_l))\times\cdots\times\textup{\Cl}(K(\zeta_l))$.
\end{corollary}

\begin{remark}\label{leoremark}
 Assume $G=C_l$ (so that $\Delta\cong (\Z/l\Z)^*$) and let $K$ be a number field. When $K$ is disjoint from $\Q(\zeta_l)$, there is a simple consequence of the above
 corollary: $K$ is $C_l$-Leopoldt if and only if $\mathcal{J}$ annihilates $\Cl(K(\zeta_l))$, where the module structure of $\Cl(K(\zeta_l))$ over $\Z[\Delta]$
 is the classical one from Galois theory.
 
 More generally, let $H\cong\Gal(K(\zeta_l)/K)$ be a proper subgroup of $\Delta=\{\sigma_1,...,\sigma_{l-1}\}$. We can identify
 $H$ with the subset of the $\sigma_j$'s such that there exists an automorphism of $\Gal(K(\zeta_l)/K)$ which sends $\zeta_l\mapsto \zeta_l^j$. 
 Looking at the structure of $K[G]\cong K[z]/(z^l-1)$, the set of the $(l-1)/d$ components of $\Cl^0(\mathcal{M})$, which are equal to $\Cl(K(\zeta_l))$, where $d=|H|$,
 is permuted via a free action of $\Delta/H$, while the module structure of $\Cl(K(\zeta_l))$ over $\Z[H]$ given by restriction is the classical one,
 namely with the identification $H\cong \Gal(K(\zeta_l)/K)$.
 
 Hence if $K$ is $C_l$-Leopoldt then $\pi_H(\mathcal{J})$ annihilates $\Cl(K(\zeta_l))$, where $\pi_H:\Z[\Delta]\rightarrow \Z[H]$ is the map
 that only conserves the coefficients of the elements of $H$.
 One can easily verify that $\pi_H(\mathcal{J})$ is an ideal of $\Z[H]$ (although $\pi_H$ is not a ring homomorphism). 
 The work by Ichimura and Sumida-Takahashi on the image of the Stickelberger ideal through this map will turn out to be very useful
 for our purposes.
\end{remark}

\section{$C_2$-Leopoldt fields}
In this situation for every field $K$ the component of the maximal order in $K[C_2]$ annihilated by the augmentation (let us call it zero-component) is $\Oc_K$. We need to know the Stickelberger ideal of $\Z$, but
in this case $\theta=\frac{1}{2}$
and so $1\in\mathcal{J}$. Therefore Corollary \ref{leocond} tells us that the class group has to be trivial and we have the following

\begin{proposition}
 A number field $K$ is $C_2$-Leopoldt if and only if $h_K=1$.
\end{proposition}

\section{$C_3$-Leopoldt fields}
Here we have to study the Stickelberger ideal of $\Z[C_2]$. But this is the unit ideal as before, because $1=(2j-1)\theta$.

If $\sqrt{-3}\in K$, then the zero-component of the maximal order in $K[C_3]$ is $\Oc_K\times \Oc_K$, so as before we have:

\begin{proposition}
 A number field $K$ with $\sqrt{-3}\in K$ is $C_3$-Leopoldt if and only if $h_K=1$.
\end{proposition}

In the other case, the zero-component of the maximal order is necessarily $\Oc_{K(\sqrt{-3})}$, so we obtain the following general proposition:

\begin{proposition}
A number field $K$ is $C_3$-Leopoldt if and only if $h_{K(\sqrt{-3})}=1$.
\end{proposition}

When $K$ is an abelian extension $K(\sqrt{-3})$ is an imaginary abelian extension of $\Q$, and from Yamamura \cite{Yam94} we know the imaginary abelian number fields 
with class number $1$.

\begin{corollary}
 The abelian Leopoldt fields of the type $C_3$ are finite in number.
\end{corollary}

\begin{remark}
 Taking $K$ general, looking at the Hilbert class field of $K$ we conclude that $h_K$ can be $1$ or $2$. We cannot say that 
 $h_K=1$, like the $C_2$ case: for example we know from \cite{Yam94} that $\Q(\sqrt{3},\sqrt{-7},\sqrt{-3})$ has class number $1$, while $\Q(\sqrt{3},\sqrt{-7})$ does not
 (and so the class number is $2$). However the conclusion holds whenever $K$ is not ramified in $3$ or it is not totally imaginary, because there is total ramification (in 
 primes over $3$ or
 the real archimedean ones) and so $h_K=1$ by \cite[Theorem 10.1]{Was97}.
\end{remark}

There is one more situation in which we have a finiteness result on $C_3$-Leopoldt fields. Using \cite{Odl75}, which states that there are
finitely many normal CM-fields with relative
class number $1$, we have:
\begin{corollary}
 There is only a finite number of normal and totally real fields which are $C_3$-Leopoldt.
 The same holds for normal CM-fields which are $C_3$-Leopoldt.
\end{corollary}

If we assume the generalized Riemann hypothesis, by \cite{AD03} we find a finite number of totally real or CM-fields as well. That article will be very useful
again at a later point in the present paper.

\section{Totally real $C_l$-Leopoldt fields}\label{real}

We define $\Delta_l:=\Gal(\Q(\zeta_l)/\Q)\cong (\Z/l\Z)^*$, with the elements $\{\sigma_1,\sigma_2,...,\sigma_{l-1}\}$, and 
$\theta_l$, $\mathcal{J}_l$ respectively
the Stickelberger element and ideal in $\Q[\Delta_l]$.

We can state a general result 
that holds when a number field $K$ is disjoint from $\Q(\zeta_l)$, using Iwasawa's class number formula \cite{Iwa62}
in a similar way to what happens in \cite{Gre97}.

\begin{proposition}
 If $K$ is a $C_l$-Leopoldt number field disjoint from $\Q(\zeta_l)$, denote 
 by $J\in\textup{\Gal} (K(\zeta_l)/K)$ the automorphism corresponding to the complex conjugation
 $\sigma_{l-1}\in\textup{\Gal}(\Q(\zeta_l)/\Q)$. Then 
 the exponent of 
 $\textup{\Cl}(K(\zeta_l))^{1-J}$ divides $h^-_l$.
\end{proposition}

\begin{proof}
 In this case we can see the class group $\Cl(K(\zeta_l))$ as a $\Z[\Delta]$-module with the identification $\Gal(\Q(\zeta_l)/\Q)\cong \Gal(K(\zeta_l)/K)$. 
 Moreover we can consider $\Cl(K(\zeta_l))^{1-J}$
as a $\Z[\Delta]^-$-module, where of course
 \begin{equation*}
  \Z[\Delta]^-:=\text{Ker}\left((1+\sigma_{-1}):\Z[\Delta]\longrightarrow\Z[\Delta]\right)=\text{Im}\left((1-\sigma_{-1}):\Z[\Delta]\longrightarrow\Z[\Delta]\right).
 \end{equation*}  
 Iwasawa's class number formula tells us that
 \begin{equation*}
 [\Z[\Delta]^-:\mathcal{J}_l^-]=h^-_{\Q(\zeta_{l^n})},
 \end{equation*}
 where $\mathcal{J}_l^-=\mathcal{J}_l\cap \Z[\Delta]^-$.

By Corollary \ref{leocond} we are done.
\end{proof}

We are more interested in the case when $K$ is totally real. In fact in this situation the automorphism $J$ continues to be a complex conjugation,
as it fixes $K\Q(\zeta_l)^+$, where
$\Q(\zeta_l)^+=\Q(\zeta_l+\zeta_l^{-1})$ is the maximal real subfield of $\Q(\zeta_l)$. Note that $K\Q(\zeta_l)^+$ is the maximal real subfield of $K(\zeta_l)$.
We are more interested in taking the minus part of the odd part:
\begin{equation*}
 \Cl(K(\zeta_l))^-:=\text{Ker}\left((1+J):\Cl(K(\zeta_l))_{odd}\longrightarrow\Cl(K(\zeta_l))_{odd}\right);
\end{equation*}
of course $\Cl(K(\zeta_l))^-$ continues to have a structure of $\Z[\Delta]^-$-module. Together with 
the well-known fact that the order of $\Cl(K(\zeta_l))^-$ is the odd part of $h^-_{K(\zeta_l)}$, this permits us to conclude the following

\begin{corollary}\label{minusresweak}
 If $K$ is a totally real $C_l$-Leopoldt number field disjoint from $\Q(\zeta_l)$, the exponent of 
 $\textup{\Cl}(K(\zeta_l))^-$ (we are taking the minus part of the odd part) divides $h^-_l$.
 In particular the odd primes dividing the relative class number 
 $h^-_{K(\zeta_l)}$ divide $h^-_l$.
\end{corollary}

In the case $l\equiv 3\pmod 4$ we have that $K(\zeta_l)/K(\sqrt{-l})$ is an odd Galois extension of CM-fields,
and by \cite[Theorem 5]{LOO97} we deduce that $h_{K(\sqrt{-l})}^-\mid h_{K(\zeta_l)}^-$.
More precisely (even though forgetting about the prime $2$), we can easily state and prove the following, which was surely already known:

\begin{proposition}\label{oddCM}
 Let $L/K$ be an odd extension of CM-fields. Then the map $N^-_{L/K}:\textup{\Cl}(L)^-\rightarrow \textup{\Cl}(K)^-$
 induced by the norm is surjective. In particular the odd part
 of $h_K^-$ divides the odd part of $h_L^-$.
\end{proposition}
\begin{proof}
First of all we describe the norm map between the minus parts: the norm map between ideal class groups
commutes with the conjugation $J$, as for every pair of primes $\mathcal{P}|\wp$ of $L/K$ we have
\begin{equation*}
 N_{L/K}(\mathcal{P})^{J}=\wp^{f(\mathcal{P}|\wp)J}=(\wp^J)^{f(\mathcal{P}|\wp)}=(\wp^J)^{f(\mathcal{P}^J|\wp^J)}=N_{L/K}(\mathcal{P}^J);
\end{equation*}
moreover, $J$ commutes with every immersion of $L$ over $K$ for the same reason for which the conjugation is independent from the immersion in $\C$ of the CM-fields.
So if $[I]^{1+J}=0$, i.e. $I^{1+J}=(x)$, then
\begin{equation*}
 N_{L/K}([I])^{1+J}=N_{L/K}([I]^{1+J})=[N_{L/K}(x)]=0,
\end{equation*}
hence the norm map is also well defined between the minus parts of the odd parts.

 Now, considering also $N_{L/K}:\Cl(L)\rightarrow \Cl(K)$, we can verify that $\coker(N_{L/K})^-=\coker(N^-_{L/K})$.  
 By the well-known commutative the diagram
given by Artin maps, namely
\[\begin{tikzcd}
\textup{\Cl}(E) \arrow{r}{\psi_{H_E/E}} \arrow[swap]{d}{N_{E/F}} & \textup{\Gal}(H_E/E) \arrow{d}{\textup{\text{ restriction}}} \\
\textup{\Cl}(F) \arrow{r}{\psi_{H_F/F}} & \textup{\Gal}(H_F/F)
\end{tikzcd}
\]
for any extension $E/F$ of number fields (see for instance the appendix in \cite{Was97} on class field theory),
 we deduce that $\coker(N_{L/K})\cong \Gal(L\cap H_K/K)$ and it is already of odd order; moreover,
 $J$, which is the complex conjugation, acts in the same way on the two groups. So $J$ acts trivially, i.e. there cannot be any automorphism of $L\cap H_K$
 in the kernel of $1+J$, whence we get $\coker(N^-_{L/K})=0$.
 
\end{proof}

Hence we get:

\begin{corollary}\label{minusres}
 If $K$ is a totally real $C_l$-Leopoldt number field disjoint from $\Q(\zeta_l)$ and $l\equiv 3\pmod 4$,
 the exponent of 
 $\textup{\Cl}(K(\sqrt{-l}))^-$ divides $h^-_l$.
 In particular the odd primes dividing the relative class number 
 $h^-_{K(\sqrt{-l})}$ divide $h^-_l$.
\end{corollary}

This is already sufficient to compute some cases of non-$C_l$-Leopoldt fields, but we can obtain some general information thanks to an idea
by Ichimura in \cite{Ich16}. In particular, our aim is to study totally real normal fields disjoint from $\Q(\zeta_l)$ which are 
$C_l$-Leopoldt, when $l=7,11,19,43,67,163$. The approach is similar to Ichimura's, with two major differences:

\begin{itemize}
 \item we want to study Leopoldt fields and not only Hilbert-Speiser fields (indeed the approach followed by Ichimura is flexible enough to apply also
         to this more general task);
 \item we correct a mistake in the article: in the proof we need the image of the Stickelberger ideal in the group ring given by the quadratic subfield of $\Q(\zeta_l)$
 instead of its own Stickelberger ideal.
\end{itemize}

Let $\delta_l:=\Gal(\Q(\sqrt{-l})/\Q)\cong \Z/2\Z$ with elements $\{1,j\}$.
If $l$ were not ramified in $K$, the extension $K(\zeta_l)/K(\sqrt{-l})$ would satisfy the hypotheses of 
\cite[Theorem 10.1]{Was97}, and so the norm map between 
class groups would be surjective. This is what Ichimura considered (indeed we already know that totally real $C_l$-Hilbert-Speiser fields
are not ramified in $l$,
by \cite{GJ09}), but actually for our purposes we can just assume non-arithmetic disjointness: by the commutative diagram
given by Artin maps, namely
\[\begin{tikzcd}
\textup{\Cl}(E) \arrow{r}{\psi_{H_E/E}} \arrow[swap]{d}{N_{E/F}} & \textup{\Gal}(H_E/E) \arrow{d}{\textup{\text{ restriction}}} \\
\textup{\Cl}(F) \arrow{r}{\psi_{H_F/F}} & \textup{\Gal}(H_F/F)
\end{tikzcd}
\]
for any extension $E/F$ of number fields, we
know that the order of the cokernel of $N:=N_{K(\zeta_l)/K(\sqrt{-l})}$ divides $(l-1)/2$, and for our aims this will be sufficient.

 If we denote by $r:\Z[\Delta_l]\rightarrow \Z[\delta_l]$
 the linear extension of the map $\Delta_l \twoheadrightarrow \delta_l$ given by
         restriction,
it is straightforward that, for $[I]\in\Cl(K(\zeta_l))$ and $\sigma\in\Z[\Delta_l]$, $N([I]^\sigma)=N([I])^{r(\sigma)}$.
This, together with Corollary \ref{leocond}, permits us to conclude the following
\begin{lemma}
  If $K$ is $C_l$-Leopoldt with $l\equiv 3\pmod 4$ then $r(\mathcal{J}_l)$ annihilates a subgroup of $\textup{\Cl}(K(\sqrt{-l}))$ with index dividing $(l-1)/2$. 
\end{lemma}

Now we want to find useful elements in $r(\mathcal{J}_l)$, in particular rational integers; in fact, if $n\in r(\mathcal{J}_l)\cap \Z$, then by the 
above lemma $n(l-1)/2$ annihilates $\Cl(K(\sqrt{-l}))$ and this will be very useful.

\begin{lemma}
 If $l\equiv 3\pmod 4$ is a prime number, then
 \begin{equation*}
  \frac{l-1}{2}h(-l)=\left(\frac{1}{l}\sum_{\left(\frac{i}{l}\right)=-1}i\right)^2-\left(\frac{1}{l}\sum_{\left(\frac{i}{l}\right)=1}i\right)^2\in r(\mathcal{J}_l)
 \end{equation*}
  where in the sums we are taking the representatives modulo $l$ between $0$ and $l-1$.
\end{lemma}
\begin{proof}
 We know from \cite[Lemma 6.9]{Was97} that $\mathcal{J}=\mathcal{J}'\theta$ (we omit subscripts),
 where $\mathcal{J}'$ is the ideal in $\Z[\Delta]$ generated by the elements of the type $c-\sigma_c$ with $(c,l)=1$.
 We note that  
 \begin{equation*}
  r(\theta)=\frac{1}{l}\sum_{\left(\frac{i}{l}\right)=1}i+\left(\frac{1}{l}\sum_{\left(\frac{i}{l}\right)=-1}i\right)j\in\Z[\delta].
 \end{equation*}
 
 Moreover, $l=l+1-\sigma_{l+1}$ and $r(4-\sigma_4)=3$. So there exist $a,b\in\Z$ such that $r(a(4-\sigma_4)+b(l+1-\sigma_{l+1}))=1$ (we assume $l\neq 3$); therefore
 $r(\theta)\in r(\mathcal{J})$. Multiplying $r(\theta)$ with its conjugate, we obtain the second equality. The first one comes from the class number formula, for example
 explained in \cite[Theorem 4.17]{Was97}.
\end{proof}

Putting the two lemmas together, we finally obtain:

\begin{corollary}\label{res}
 If $l\equiv 3\pmod 4$ is a prime number and $K$ is a $C_l$-Leopoldt field disjoint from $\Q(\zeta_l)$, then
 $\textup{\Cl}(K(\sqrt{-l}))$ has exponent divisible by $\frac{(l-1)^2}{4}h(-l)$. In particular 
 $h_{K(\sqrt{-l})}$ is only divisible by primes dividing
 $\frac{l-1}{2}h(-l)$.
 
 Note that when $h(-l)=1$ we have that $\frac{l-1}{2}h(-l)$ is odd.
\end{corollary}

For a moment we forget about the condition $l=7,11,19,43,67,163$.
The last corollary, modulo the generalized Riemann hypothesis, ensures us that for every $l\equiv 3\pmod 4$ there are
only finitely many number fields $K$ disjoint from $\Q(\zeta_l)$
such that $K(\sqrt{-l})$ is CM (i.e. $K$ is totally real or CM) and $K$ is $C_l$-Leopoldt: in fact Amoroso and Dvornicich in
\cite{AD03} proved that conditionally the exponent of the ideal class group of a CM-field tends to infinity with its discriminant.

We put it on record:
\begin{corollary}
 Assuming the generalized Riemann hypothesis, if $l\equiv 3\pmod 4$ there is only a finite number of totally real and CM fields disjoint from $\Q(\zeta_l)$
 that are $C_l$-Leopoldt.
\end{corollary}

If we find conditions for the class number to be $1$ (more precisely, we are going to study the relative class number), we get unconditional
finiteness properties for $C_l$-Leopoldt fields. This is why we assume $l=7,11,19,43,67,163$, as we will see.

\begin{remark}
 We briefly explain why we used the quadratic subfield of $K(\zeta_l)/K$. In Ichimura's article, the author claimed that in our cases $r(\mathcal{J})=\Z[\delta]$ looking
 at Sinnott's result in \cite{Sin80}; however this is not true (for example because the elements that come from the Stickelberger ideal, which in cyclotomic
 fields fits also Sinnott's definition,
 have too big augmentation) and the problem is that Sinnott's result concerns the Stickelberger ideal of the quadratic field,
 and not the the image from the cyclotomic one. So we are unable to prove that $1\in r(\mathcal{J})$, 
 but we do find a certain natural number, and it will be very important to know that it is odd. This 
 would not have been possible if we only considered the cyclotomic extension, because by Lemma 2.1 of Sinnott's article
 the elements of the Stickelberger ideal live in the set
 \begin{equation*}
  A:=\left\{\sigma\in\Z[\Delta]:(1+j)\sigma\in(\sigma_1+...+\sigma_{l+1})\Z\right\}.
 \end{equation*}
 In particular there are no non-zero elements in $\Z$ that are also in the Stickelberger ideal.
\end{remark}

Now we can put together our last result with Corollary \ref{minusres},
so that the case of $h(-l)=1$ will reduce our task to check a finite number of fields.

First of all, looking at any class number table (for example from \cite{Was97}), we note the following

\begin{lemma}
 If $h(-l)=1$, then $\left(h^-_l,\frac{l-1}{2}\right)=1$.
\end{lemma}

Hence we have finally obtained that:

\begin{theorem}\label{true}
 Let $K$ be a totally real $C_l$-Leopoldt field disjoint from $\Q(\zeta_l)$, where $h(-l)=1$. Then $h_{K(\sqrt{-l})}^-=1$. In particular,
 normal totally real $C_l$-Leopoldt fields unramified in $l$ form a finite set because of \textup{\cite{Odl75}}
 ($K(\sqrt{-l})$ is CM).
\end{theorem}
 
Of course we can find from the finite list of $C_l$-Leopoldt fields also the totally real
abelian $C_l$-Hilbert-Speiser fields, and this is a way to adjust Ichimura's result: it is still true that there are finitely many
totally real abelian $C_l$-Hilbert-Speiser fields.

\begin{remark}
 As we have the divisibility $h(-l)\mid h_l^-$, this method does not work if $h(-l)>1$.
\end{remark}

\section{Complete list of real abelian $C_l$-Hilbert-Speiser fields for $l=7,11,19,43,67,163$}\label{hilbert-speiser}
Here we prove that the list of totally real abelian $C_l$-Hilbert-Speiser fields for $l=7,11,19,43,67,163$ given in \cite{Ich16} is actually complete.
We have to study the discrepancy between Theorem \ref{true} and \cite[Corollary]{Ich16}: since the latter simply claims that 
$h_{K(\sqrt{-l})}=1$, we have to check the fields such that $h_{K(\sqrt{-l})}^-=1$ and $h_K=h_{K(\sqrt{-l})^+}\neq1$; moreover, as we are interested in
Hilbert-Speiser fields, we already know that $K$ must be unramified in $l$.

Unfortunately there is no article with a complete table of the imaginary abelian fields with relative class number $1$. The article \cite{CK00} gives the non-cyclic ones,
while the list of cyclic imaginary fields is slightly more difficult to reconstruct: we have to look at the tables in \cite{PK97}, \cite{PK98} and \cite{CK98}.

Here is our analysis:
\begin{itemize}
 \item Imaginary cyclic sextic number fields: the list of sextic fields with relative class number $1$ can be consulted in table $3$ of \cite{PK97}.
 They have the following curious property: every such field $L$ with class number different from $1$ (i.e. $h_{L^+}> 1$) has the same conductor
 of $L^+$. Thus, if $K$ is $C_l$-Hilbert-Speiser and $L=K(\sqrt{-l})$ belongs to this list, $K$ is surely ramified in $l$ since $L$ is. So we do not
 gain anything new.
 \item Imaginary cyclic number fields of $2$-power degree: \cite{PK98} shows that these fields have degree up to $16$ and gives a complete list 
 of the non-quadratic ones (since we can assume $K\supsetneq\Q$, we are not interested in them). From the tables we see that 
 if a cyclic imaginary field of degree $4$, $8$ or $16$ has relative class number $1$, then it has class number $1$.
 \item Other imaginary cyclic number fields: in \cite{CK98} the authors showed that they have degree less or equal than $20$, and they are listed in Table I.
 There is only one such field with class number different from $1$, but both the conductors of $L$ and $K=L^+$ are equal to $91$. Like before, we
 obtain that this is not a new case of a quadratic extension of a Hilbert-Speiser field.
 \item Non-cyclic imaginary abelian number fields: the complete list is given in \cite[Table I]{CK00}. Here there are a lot
 of fields with class number greater than $1$, and unfortunately the authors do not give conductors of the fields, but even in this case
 it is not difficult to conclude that every totally real $C_l$-Hilbert-Speiser $K$ with $K(\sqrt{-l})$ in this table and $h_K>1$ is ramified in $l$.
 In fact almost all the fields $L$ in the table which contain $\sqrt{-l}$ for $l=7,11,19,43,67,163$ are of the form $F(\sqrt{-l})$, with 
 $F$ imaginary and unramified in $l$ (this means that its characters, that can be read from the table, are not all even and have conductors coprime to $l$);
 so their maximal real subfield must be ramified in $l$ by an easy argument. All the exceptions concern $l=7$: $\langle \chi_3,\chi_7^3,\chi_7^4\psi_9\rangle$,
 $\langle \chi_7^3,\chi_3\chi_5^2,\chi_7^2\rangle$ and $\langle \chi_7^3,\chi_7^3\chi_{13}^6,\chi_7^4\chi_{13}^4\rangle$, for which 
 $\langle\chi_7^4\psi_9\rangle$, $\langle\chi_7^2\rangle$ and $\langle\chi_7^4\chi_{13}^4\rangle$ respectively are real and
 hence contained in $K$ and ramified at $7$.
\end{itemize}

\begin{corollary}
 The only real abelian $C_7$-Hilbert-Speiser fields are $\Q(\sqrt{5})$ and $\Q(\sqrt{13})$. The only real abelian $C_{11}$-Hilbert-Speiser field
 is $\Q(\cos(2\pi/7))$. There is no real abelian $C_l$-Hilbert-Speiser field if $l=19,43,67,163$.
\end{corollary}

\section{$C_l$-Leopoldt fields that intersect $\Q(\zeta_l)$}\label{intersection}
Until now we only studied the property of being $C_l$-Leopoldt of fields that are disjoint from $\Q(\zeta_l)$. Using the observations
in Remark \ref{leoremark}, we can now concentrate on number fields which intersect $\Q(\zeta_l)$.

Suppose initially that $\zeta_l\in K$. Then the group $\Delta=\Delta_l=\Gal(\Q(\zeta_l)/\Q)$ acts freely on the set of the $l-1$ components of 
\begin{equation*}
 \Cl^0(\mathcal{M})\cong\Cl(K)\times\cdots\times\Cl(K). 
\end{equation*}
Since $l\theta\in\mathcal{J}$ and annihilates $\Cl^0(\mathcal{M})$, for
every $[I]\in \Cl(K)$ we know from Corollary \ref{leocond} that $l\theta$ annihilates 
\begin{equation*}
 ([I],0,...,0)\in\Cl(K)\times\cdots\times\Cl(K).
\end{equation*}
But the first component of $([I],0,...,0)^{l\theta}$ is $[I]$ itself, because the coefficient of $1$ in $l\theta\in\Z[\Delta]$ is $1$. Hence we have obtained:
\begin{proposition}
 Let $K$ be a $C_l$-Leopoldt field that contains $\zeta_l$. Then $\textup{\Cl}(K)=\{1\}$, i.e. $h_K=1$. Of course the converse holds.
\end{proposition}

A simple consequence of the above result is that there are no CM $C_l$-Leopoldt fields which are odd extensions of $\Q(\zeta_l)$ if $l\geq 23$, by
\cite[Theorem 5]{LOO97} (in odd extensions $E/F$ of CM-fields there is the divisibility $h_F^-\mid h_E^-$). Moreover, assuming the generalized Riemann hypothesis, 
we conclude that there are only finitely many CM $C_l$-Leopoldt fields that contain $\zeta_l$, from \cite{AD03}.

In particular, $\Q(\zeta_l)$ is 
$C_l$-Leopoldt if and only if $2\leq l\leq 19$. We also observe that the last proposition implies \cite[Proposition 3.3]{Her05}, apart from
a finite number of $l$'s.

Now we consider the situation in which $K\cap \Q(\zeta_l)=\Q(\zeta_l)^+=\Q(\zeta_l+\zeta_l^{-1})$. 
Here the decomposition is 
\begin{equation*}
 \Cl^0(\mathcal{M})\cong\Cl(K(\zeta_l))\times\cdots\times\Cl(K(\zeta_l)),
\end{equation*}
where
we have $(l-1)/2$ copies of the ideal class group of $K(\zeta_l)$. From what we have already said in Remark \ref{leoremark}, each
component has a structure of $\Z[H]$-module, where 
\begin{equation*}
 H=\Gal(K(\zeta_l)/K)\cong \Gal(\Q(\zeta_l)/\Q(\zeta_l)^+)=\{1,\sigma_{-1}\}.
\end{equation*}
Since $\pi_H(l\theta)=1+(l-1)\sigma_{-1}$ and $\pi_H((2-\sigma_2))\theta=\sigma_{-1}$, then $\pi_H(l\theta-(l-1)(2-\sigma_2)\theta)=1$ and annihilates
$\Cl(K(\zeta_l))$ if $K$ is $C_l$-Leopoldt. Therefore we have an analogous conclusion:

\begin{proposition}
 Let $K$ be a number field such that $K\cap \Q(\zeta_l)=\Q(\zeta_l)^+$. Then $K$ is $C_l$-Leopoldt if and only if $h_{K(\zeta_l)}=1$.
\end{proposition}

\begin{remark}
 The above results are exactly the generalizations of our criteria about $C_2$ and $C_3$-Leopoldt fields.
 Moreover, modulo the generalized Riemann hypothesis and finitely many cases this result implies \cite[Theorem 1.1]{GJ09},
 when $[K(\zeta_l):K]=2$.
\end{remark}

At this point it is
tempting to look at the case $|H|=3$; and indeed, the conclusion continues to hold,
as explained in the following result:

\begin{proposition}
 Let $l\equiv 1\pmod 3$ be a prime and $K$ a number field with the property that $K\cap \Q(\zeta_l)$ is the subfield of $\Q(\zeta_l)$ with degree $(l-1)/3$
 over $\Q$. Then $K$ is $C_l$-Leopoldt if and only if $h_{K(\zeta_l)}=1$.
\end{proposition}
\begin{proof}
 In a similar way to the preceding cases, we want to prove that $\pi_H(\mathcal{J})=\Z[H]$. Let $a,b$ be the non-trivial representatives 
 of the third roots of $1$ modulo $l$, i.e. such that $H=\{\sigma_1,\sigma_a,\sigma_b\}$. For every $2\leq c\leq l-1$ we have 
 \begin{equation*}
  \pi_H((c-\sigma_c)\theta)=\left\lfloor\frac{ca}{l}\right\rfloor\sigma_a^{-1}+\left\lfloor\frac{cb}{l}\right\rfloor\sigma_b^{-1}=
  \left\lfloor\frac{cb}{l}\right\rfloor\sigma_a+\left\lfloor\frac{ca}{l}\right\rfloor\sigma_b.
 \end{equation*}
 
 If there exists $2\leq c\leq l-1$ such that $\left\lfloor\frac{ca}{l}\right\rfloor\neq\left\lfloor\frac{cb}{l}\right\rfloor$, taking the minimum $c_m$
 of such values, we notice that $\pi_H((c_m-\sigma_{c_m})\theta-(c_m-1-\sigma_{c_m-1})\theta)$ is $\sigma_a$ or $\sigma_b$,
 since every increase inside the integral part is $a/l<1$ and $b/l<1$ when $c$ rises by $1$, whence we obtain
 that $1\in\pi_H(\mathcal{J})$ as $\pi_H(\mathcal{J})$ is an ideal.
 
 Therefore we may assume that for every $2\leq c\leq l-1$ we have $\left\lfloor\frac{ca}{l}\right\rfloor=\left\lfloor\frac{cb}{l}\right\rfloor$. Taking
 $c=l-1$ this implies $a-1=b-1$, that is $a=b$, which is not possible.
 
\end{proof}

Our considerations about the projection of the Stickelberger ideal are also consequences of \cite[Theorem 2(III)]{IS06},
in a work by Ichimura and Sumida-Takahashi concerning $p$-normal integral bases. According to Lemma 1
of the same article, when $|H]$ is even, then $\pi_H(\mathcal{J})$ is contained in the ideal of $\Z[H]$ generated by $1+\rho+\rho^2+\cdots+\rho^{|H|/2-1}$,
where $\rho$ is any generator of $H$, and so cannot be the whole ring if $|H|\geq 4$ (e.g. look at the augmentation).

If $|H|$ is odd, then \cite[Theorem 2(I)]{IS06} and \cite[Theorem 2(II)]{IS06} imply
that the order of the quotient $\Z[H]/\pi_H(\mathcal{J})$ divides 
\begin{equation*}
 \left[(1+\rho+\rho^2+\cdots+\rho^{(l-1)/2-1})\Z[\Delta]:\pi_\Delta(\mathcal{J})\right]=h_l^-,
\end{equation*}
where $\rho$ generates $\Delta$, so:
\begin{corollary}\label{corollarygain}
 Let $K$ be a $C_l$-Leopoldt field such that the degree $[\Q(\zeta_l):K\cap \Q(\zeta_l)]$ is odd.
 Then the exponent of $\textup{\Cl}(K(\zeta_l))$ divides $h_l^-$. In particular if $3\leq l\leq 19$ then $K$ is $C_l$-Leopoldt if and only if $h_{K(\zeta_l)}=1$.
\end{corollary}

Remaining in the situation of $|H|$ being odd, we can say something more precise. Now we start from above, and assume $l\equiv 3\pmod 4$ and $|H|=(l-1)/2$, 
i.e. $K\cap\Q(\zeta_l)=\Q(\sqrt{-l})$. In this case Ichimura's class number formula in \cite[Theorem]{Ich06a} tells us that 
$[\Z[H]:\pi_H(\mathcal{J})]=h_l^-/h(-l)$. According to \cite[Theorem 2(II)]{IS06},
thanks to which we know that the index of $\pi_{H_1}(\mathcal{J})$ in $\Z[H_1]$ 
divides that concerning $H_2$ whenever $H_1<H_2$ are of odd order, this permits us to conclude
the following
\begin{corollary}
 Let $K$ be a $C_l$-Leopoldt field such that the degree $[\Q(\zeta_l):K\cap \Q(\zeta_l)]$ is odd, with $l\equiv 3\pmod 4$.
Then the exponent of $\textup{\Cl}(K(\zeta_l))$ divides $h_l^-/h(-l)$. 
\end{corollary}
In particular, if $l=7,11,19,23$ then $K$ is $C_l$-Leopoldt if and only if $h_{K(\zeta_l)}=1$.
Moreover in a lot of other cases we gain some information on divisibility respect to Corollary \ref{corollarygain},
i.e. we are able to find primes that have the same exponent
in $h(-l)$ and $h_l^-$; for example, the first cases for $l$ are $47$, $59$, $71$, $79$, $103$, $107$, $127$, $151$, $167$, $179$, $191$, $223$, $239$.

When $|H|$ is even there is no hope to obtain in this way an equivalence between being $C_l$-Leopoldt and divisibility properties of $h_{K(\zeta_l)}$:
the product
\begin{equation*}
 (a_0+a_1\rho+\cdots+a_{|H|-1}\rho^{|H|-1})\cdot (1+\rho+\rho^2+\cdots+\rho^{|H|/2-1})
\end{equation*}
cannot be a non-zero rational integer for any $a_0+a_1\rho+\cdots+a_{|H|-1}\rho^{|H|-1}\in\Z[H]$ by an easy polynomial divisibility argument.
Hence we do not find rational integers that annihilate $\Cl(K(\zeta_l))$. However it is possible to get a conclusion in a similar fashion to Corollary
\ref{minusresweak}: since $(1-\rho)(1+\rho+\rho^2+\cdots+\rho^{|H|/2-1})=1-\rho^{|H|/2}$ we have that the index 
$n_H:=[(1+\rho+\rho^2+\cdots+\rho^{|H|/2-1})\Z[H]:\pi_{H}(\mathcal{J})]$ is such that $n_H\Z[H]^-\subseteq \pi_{H}(\mathcal{J})$, where the minus is referred to
``conjugation``, that is the automorphism $J:=\rho^{|H|/2}\in H$ which corresponds to the conjugation $\sigma_{-1}\in\Delta$. Moreover by
Ichimura and Sumida-Takahashi's work $n_H|h_l^-$. Hence:
\begin{proposition}
 Let $K$ be a $C_l$-Leopoldt field such that $H=\Gal(K(\zeta_l)/K)$ has even order. 
 Then $n_H$ annihilates $\textup{\Cl}(K(\zeta_l))^{1-J}$. In particular $h_l^-$ does.
\end{proposition}

As in Corollary \ref{minusresweak} we get the following
\begin{corollary}
 With the preceding hypotheses, let us assume that $K$ is totally real. Then the odd primes that divide $h_{K(\zeta_l)}^-$ also divide $h_l^-$.
\end{corollary}

Actually, without this generality we can get better results if $|H|<(l-1)/2$ and without assuming $l\equiv3\pmod4$. In fact if
$23\leq l\leq 499$, from \cite[Proposition 3]{IS06} we can read off the few cases in which $|H|<(l-1)/2$ and $\pi_H(\mathcal{J})\subsetneq \Z[H]$:
for instance if $|H|$ is odd we only have 
$l=277$ and $|H|=69$, $l=349$ and $|H|=87$, $l=331$ and $|H|=33$. The authors expected in general that this situation is very rare, i.e. 
we can conclude as well that in most cases a $C_l$-Leopoldt field with $|H|<(l-1)/2$ odd is such that $h_{K(\zeta_l)}=1$. 
\cite[Theorem 2]{Ich06b} gives more evidence about this expectation: if, for $H$ of odd order, there is a prime $q|[\Z[H]:\pi_H(\mathcal{J})]$, then
$h_l^-$ is divisible by $q^{[\Delta:H]/2}$; on the contrary, it is quite common for a cyclotomic relative class number to have rather small exponents.
This is also true in the even case substituting $\Z[H]$ with $(1+\rho+\rho^2+\cdots+\rho^{|H|/2-1})\Z[H]$, but
as we have seen we get a much weaker annihilation/divisibility conclusion.

To sum up: in the context of
$C_l$-Leopoldt fields $K$ that intersect nontrivially with $\Q(\zeta_l)$, if we just use
ad hoc arguments and previous results of Ichimura and Sumida-Takahashi it appears that we are
able to obtain restrictions on the class number of $\Q(\zeta_l)$  only in the
case when $[\Q(\zeta_l):K\cap \Q(\zeta_l)]$ is equal to $2$ or an odd number, and thus finiteness conditions
if we assume the generalized Riemann hypothesis, by \cite{AD03} (without fixing $l$ if $[\Q(\zeta_l):K\cap \Q(\zeta_l)]$ is $1$, $2$ or $3$).
Returning to Hilbert-Speiser fields, from \cite{Her05} we already knew that we
do not have such odd cases of $C_l$-Hilbert-Speiser fields for $l\geq 5$, unless the intersection 
is $\Q(\sqrt{-l})$ when $l\equiv 3\pmod4$. When instead $[\Q(\zeta_l):K\cap \Q(\zeta_l)]$ is even, i.e. $K\cap \Q(\zeta_l)$ is real,
then it is likely that if $K$ is $C_l$-Hilbert-Speiser then the intersection is trivial, by \cite{Ich07}.

\section{Acknowledgements}
 The majority of the results were
obtained in the first author's M.Sc. thesis, written under the joint
supervision of
Ilaria Del Corso at Pisa and the second author.


\begin{thebibliography}{9} 
  \bibitem[AD03]{AD03}
     Francesco Amoroso and Roberto Dvornicich,
     \emph{Lower bounds for the height and size of the ideal class group in CM-fields},
     Monatshefte für Mathematik 138 (2003), 85-94
  \bibitem[BCGJ11]{BCGJ11}
     Nigel Byott, James E. Carter, Cornelius Greither and Henri Johnston,
     \emph{On the restricted Hilbert-Speiser and Leopoldt properties},
     Illinois Journal of Mathematics 55 no. 2 (2011), 623-639
  \bibitem[Car03]{Car03}
     James E. Carter,
     \emph{Normal integral bases in quadratic and cyclic cubic extensions of quadratic fields},
     Archiv der Mathematik 81 (2003), 266-271; \emph{Erratum}, Archiv der Mathematik 83 (2004), vi-vii
  \bibitem[Car07]{Car07}
     James E. Carter,
     \emph{Some remarks on Hilbert-Speiser and Leopoldt fields of given type},
    Colloquium Mathematicum 108 no. 2 (2007), 217-223
  \bibitem[CK98]{CK98}
     Ku-Young Chang and Soun-Hi Kwon,
     \emph{Class number problem for imaginary cyclic number fields},
     Journal of Number Theory 73 (1998), 318-338
  \bibitem[CK00]{CK00}
     Ku-Young Chang and Soun-Hi Kwon,
     \emph{Class numbers of abelian number fields},
     Proceedings of the American Mathematical Society 128 no. 9 (2000), 2517-2528
  \bibitem[CR87]{CR87}
     Charles Curtis and Irving Reiner,
     \emph{Methods of representation theory. With applications to finite groups and orders. Volume II},
     Wiley-Interscience New York (1987)
  \bibitem[Gre90]{Gre90}
     Cornelius Greither,
     \emph{Relative integral normal bases in $\Q(\zeta_p)$},
     Journal of Number Theory 35 no. 2 (1990), 180-193
  \bibitem[Gre97]{Gre97}
     Cornelius Greither,
     \emph{On normal integral bases in ray class fields over imaginary quadratic fields},
     Acta Arithmetica LXXVIII.4 (1997), 315-329
  \bibitem[GJ09]{GJ09}
     Cornelius Greither and Henri Johnston,
     \emph{On totally real Hilbert-Speiser fields of type $C_p$},
     Acta Arithmetica 138.4 (2009), 329-336
  \bibitem[GJ12]{GJ12}
     Cornelius Greither and Henri Johnston,
     \emph{Non-existence and splitting theorems for normal integral bases},
     Annales de l'Institut Fourier 62 no. 1 (2012), 417-437
  \bibitem[GRRS99]{GRRS99}
     Cornelius Greither, Daniel R. Replogle, Karl Rubin and Anupam Srivastav,
     \emph{Swan modules and Hilbert-Speiser number fields},
     Journal of Number Theory 79 (1999), 164-173
  \bibitem[Her05]{Her05}
     Thomas Herreng,
     \emph{Sur les corps de Hilbert-Speiser},
     Journal de Théorie des Nombres de Bordeaux 17 (2005), 767-778
  \bibitem[Ich04]{Ich04}
     Humio Ichimura,
     \emph{Normal integral bases and ray class groups},
     Acta Arithmetica 114.1 (2004), 71-85
  \bibitem[Ich06a]{Ich06a}
     Humio Ichimura,
     \emph{A class number formula for the $p$-cyclotomic field},
     Archiv der Mathematik 87 (2006), 539-545
  \bibitem[Ich06b]{Ich06b}
     Humio Ichimura,
     \emph{Triviality of Stickelberger ideals of conductor $p$},
     Journal of Mathematical Science (University of Tokyo) 13 (2006), 617-628
  \bibitem[Ich07]{Ich07}
     Humio Ichimura,
     \emph{Note on Hilbert-Speiser number fields at a prime $p$},
     Yokohama Mathematical Journal 54 (2007), 45-53
  \bibitem[Ich16]{Ich16}
     Humio Ichimura,
     \emph{Real abelian fields satisfying the Hilbert-Speiser condition for some small primes $p$},
     Proceedings of the Japan Academy 92 (2016), 19-22
  \bibitem[IS06]{IS06}
     Humio Ichimura and Hiroki Sumida-Takahashi,
     \emph{Stickelberger ideals of conductor $p$ and their application},
     Journal of the Mathematical Society of Japan 58 no. 3 (2006), 885-902
  \bibitem[Iwa62]{Iwa62}
     Kenkichi Iwasawa,
     \emph{A class number formula for cyclotomic fields},
     Annals of Mathematics 76 (1962), 171-179
  \bibitem[LOO97]{LOO97}
     Stéphane Louboutin, Ryotaro Okazaki and Michel Olivier,
     \emph{The class number one problem for some non-abelian normal CM-fields},
     Transactions of the American Mathematical Society 349 no. 9 (1997), 3657-3678
  \bibitem[McC83]{McC83}
     Leon R. McCulloh,
     \emph{Galois module structure of elementary abelian extensions},
     Journal of Algebra 82 (1983), 102-134
  \bibitem[Odl75]{Odl75}
     Andrew M. Odlyzko,
     \emph{Some analytic estimates of class numbers and discriminants},
     Inventiones Mathematicae 29 (1975), 275-286
  \bibitem[PK97]{PK97}
     Young-Ho Park and Soun-Hi Kwon,
     \emph{Determination of all imaginary abelian sextic number fields with class number $\leq11$},
     Acta Arithmetica LXXXII.1 (1997), 27-43
  \bibitem[PK98]{PK98}
     Young-Ho Park and Soun-Hi Kwon,
     \emph{Determination of all non-quadratic imaginary cyclic number fields of $2$-power degree with relative class number $\leq20$},
     Acta Arithmetica LXXXIII.3 (1998), 211-223
  \bibitem[Sin80]{Sin80}
     Warren Sinnott,
     \emph{On the Stickelberger ideal and the circular units of an abelian field}, 
  \bibitem[Was97]{Was97}
     Lawrence C. Washington,
     \emph{Introduction to cyclotomic fields},
     Springer Science+Business Media New York Inc. (1997)
  \bibitem[Yam94]{Yam94}
     Ken Yamamura,
     \emph{The determination of the imaginary abelian number fields with class number one},
     Mathematics of Computation 62 no. 206 (1994), 899-921
  \bibitem[Yos09]{Yos09}
     Yusuke Yoshimura,
     \emph{Abelian number fields satisfying the Hilbert-Speiser condition at $p=2$ or $3$},
     Tokyo Journal of Mathematics 32 no. 1 (2009), 229-235
  \end{thebibliography}
\end{document}